\documentclass[a4paper,12pt]{amsart}
\usepackage[T1]{fontenc}
\usepackage{palatino,amsfonts,amssymb,amsmath,amsthm,xcolor,enumitem}
\usepackage[left=2.8cm,top=2.8cm,bottom=2.8cm,right=2.8cm]{geometry}
\usepackage{hyperref} 
\hypersetup{
    colorlinks,
    linkcolor=blue,
    citecolor=blue,
    urlcolor=blue
}
\numberwithin{equation}{section}
\newtheorem{thm}{Theorem}[section]
\newtheorem{cor}[thm]{Corollary}
\newtheorem{proposition}[thm]{Proposition}
\newtheorem{df}[thm]{Definition}
\newtheorem{obs}{Observation}

\newtheorem{lemma}[thm]{Lemma}

\usepackage{fancyhdr}
\title{Log-concavity and concentration bounds\\ for $\beta$-ensemble gaps}
\author{Samuel G. G. Johnston}

\begin{document}

\begin{abstract}
We consider random variables $\lambda_1^N \leq \ldots \leq \lambda_N^N$ distributed according to the general $\beta$-ensemble, that is, the probability density function on $\mathbb{R}^N$ taking the form
\begin{align*}
p_N(\lambda_1,\ldots,\lambda_N) := C_N \mathrm{1}\{ \lambda_1 \leq \cdots \leq \lambda_N \} \left( \prod_{1 \leq i < j\leq N} (\lambda_j - \lambda_i) \right)^\beta \exp \left\{ - \frac{\beta N}{2} \sum_{i=1}^N V(\lambda_i) \right\},
\end{align*}
where $V:\mathbb{R} \to \mathbb{R}$ is a convex function. {Our main observation is that} the marginal distribution of any subset of the $\lambda_i^N$ is log-concave, and that the same is true for the gaps between any consecutive $\lambda_i^N$. 
We use this observation to prove several concentration bounds for the normalised gaps
\begin{align*}
g_i^N := \frac{\lambda_{i+1}^N - \lambda_i^N}{ \mathbb{E}[ \lambda_{i+1}^N - \lambda_i^N ] }
\end{align*}
between consecutive values: moment bounds, exponential upper tail bounds and linear lower tail bounds. In the GUE case we improve on bounds recently obtained by Tao using tools from determinantal processes. We observe {furthermore that} the $\beta$-Gaudin-Mehta distribution enjoys these same properties. {Finally, under a mild spacing assumption in the bulk} (which is satisfied {whenever $V$ is $\mathcal{C}^4$ and $\beta\geq1$}) {we prove that} $\mathbb P(g_i^N\geq g)\allowbreak\leq C_\delta \exp(-c_\delta \min\{g^2,Ng\})$ for all $N \geq 1$ and $\delta N \leq i \leq (1-\delta)N$. 

\end{abstract}

\maketitle

\section*{AI Disclosure}

AI was used for light proofreading, literature search, and tidying up proofs. Substantive mathematical ideas {were the author's,} with one important exception: the proof of the bound $\mathbb P(g_i^N\geq g)\allowbreak\leq C\exp[-c\min\{g^2,Ng\}]$ hinges on a clever application of the Pr\'ekopa--Leindler inequality, which the author credits in full to AI.

\section{Introduction and background}

\subsection{$\beta$-ensembles and random matrices}
We consider random variables $\lambda_1^N \leq \ldots \leq \lambda_N^N$ distributed according to the general $\beta$-ensemble, that is, the probability density function on $\mathbb{R}^N$ taking the form
\begin{align} \label{eq:betaens}
p_N(\lambda_1,\ldots,\lambda_N) := C_N \mathrm{1}\{ \lambda_1 \leq \cdots \leq \lambda_N \} \left( \prod_{1 \leq i < j\leq N} (\lambda_j - \lambda_i) \right)^\beta \exp \left\{ - \frac{\beta N}{2} \sum_{i=1}^N V(\lambda_i) \right\},
\end{align}
where $\beta>0$ and $V$ is a convex function. We assume without further comment that $V$ is chosen so that $V(x) \geq - c_1 + c_2|x|$ for some $c_1,c_2 > 0$, which guarantees that a probability density function of this form exists.

Probability density functions of this form are well-studied across the literature. Most prominently, these densities are studied because they arise naturally in random matrix theory.
In particular, a Hermitian random matrix 
\begin{equation*}
   X = X^{(N)} = (X_{i,j})_{1 \leq i,j \leq N}
\end{equation*}
is said to be distributed according to the Gaussian unitary ensemble (GUE for short) if its density is proportional to $e^{ -N\mathrm{tr}(X^2)}$. Equivalently, the diagonal entries of this matrix are real Gaussian random variables with variance $1/2N$, and the above-diagonal entries are complex Gaussian random variables whose real and imaginary parts each have variance $1/4N$. These real Gaussian coordinates are centred and independent. After determining the above-diagonal entries, the below-diagonal entries are defined by conjugation, i.e.\ $X_{j,i } := \overline{X_{i,j}}$ for $i<j$. In this case, the eigenvalues are distributed according to the probability density function in \eqref{eq:betaens} with $\beta = 2$ and $V(\lambda) = \lambda^2$. Related random matrix models can be used to generate the cases $\beta = 1$ or $\beta = 4$ with $V(\lambda) = \lambda^2$ \cite{AGZ}.

\subsection{Macroscopic behaviour of the ensemble}

The probability density function in \eqref{eq:betaens} on {the ordered chamber} $\mathbb{R}^{N,\uparrow} := \{ \lambda \in \mathbb{R}^N : \lambda_1 \leq \cdots \leq \lambda_N \}$ may be written in terms of the random empirical measure $\mu_N := \frac{1}{N} \sum_{i=1}^N \delta_{\lambda_i^N}$ of the $\lambda_i$ via the expression
\begin{align*}
p_N(\lambda_1,\ldots,\lambda_N) = C_N \exp\left\{ - \beta N^2 {\tilde{\mathcal{F}}[ \mu_N ]}  \right\}, \qquad \mu_N := \frac{1}{N} \sum_{i=1}^N \delta_{\lambda_i}
\end{align*} 
where $\mathcal{F}[\mu]$ is the energy
\begin{align} \label{eq:calF}
\mathcal{F}[ \mu ] = {-\frac12\iint_{\mathbb R^2}\log|t-s|} \mu(\mathrm{d}s)\mu(\mathrm{d}t) + \frac{1}{2}\int_{-\infty < s < \infty } V(s) \mu(\mathrm{d}s).
\end{align}

{Here $-\log 0=+\infty$, and $\tilde{\mathcal{F}}$ denotes the same energy with the diagonal omitted from the double integral.}

Under reasonably mild conditions, if we take a sequence of random empirical measures $(\mu_N)_{N \geq 1}$, where each $\mu_N$ is distributed according to $p_N$, then the $(\mu_N)$ will converge weakly to the measure $\mu = \mu_V$ minimising $\mathcal{F}[\mu]$. In particular, in the bulk, the typical spacing between consecutive $(\lambda_i^N)$ will be of the order $1/N$. More explicitly, if $\mu$ has probability density function $\rho:\mathbb{R} \to [0,\infty)$, then the typical gap between consecutive values near a point $x$ on the real line will be of the order $1/(\rho(x)N)$. 

{By convexity of $V$ and of $-\log$ on $(0,\infty)$, $\mathcal F[\mu]$ is convex in the quantile function of $\mu$. The minimiser is unique by \cite[Lemma~2.6.2]{AGZ}.} In many natural cases, the minimiser $\mu$ of $\mathcal{F}[\mu]$ can be computed explicitly. For instance, in the case where $V(\lambda) = \lambda^2$, the minimiser $\mu$ is the semicircle law $\mu_{\mathrm{sc}}$, given by 
\begin{align*}
\mu_{\mathrm{sc}}(\mathrm{d}x) = \pi^{-1}(2-x^2)^{1/2}_+\,\mathrm{d}x.
\end{align*}

\subsection{Microscopic behaviour of the ensemble}

In the present article we are interested in studying the gaps
$\lambda_{i+1}^N-\lambda_i^N$ between consecutive values of the
$\beta$-ensemble. While some authors consider the density-normalised gaps
$\tilde{g}_i^N := N\rho(Q(i/N))(\lambda_{i+1}^N-\lambda_i^N)$, where
$\rho:\mathbb{R}\to[0,\infty)$ is the density of $\mu$ and
$Q:(0,1)\to\mathbb{R}$ is its quantile function, we will find it cleaner
to consider the expectation-normalised gaps
\begin{align} \label{eq:gapdef}
g_i^N := \frac{\lambda_{i+1}^N-\lambda_i^N}
{\mathbb{E}[\lambda_{i+1}^N-\lambda_i^N]}.
\end{align}
{For $V\in C^4(\mathbb R)$ and $\beta\geq1$, as shown in Appendix~\ref{app:expected-gaps}}, the two normalisations are asymptotically equivalent in the bulk, in the sense that
$\tilde{g}_i^N=(1+o(1))g_i^N$, or equivalently,
$N\rho(Q(i/N))\mathbb{E}[\lambda_{i+1}^N-\lambda_i^N]\to 1$
as $N\to\infty$, uniformly for
$\delta N\leq i\leq(1-\delta)N$.

To describe the point process that occurs in the bulk,
suppose that the density $\rho:\mathbb{R} \to [0,\infty)$ of the minimiser of \eqref{eq:calF}
is continuous and strictly positive at the origin.
Consider now the rescaled point process
$(N\rho(0)\lambda_i^N : 1\leq i\leq N)$, with the scaling chosen so
that the limiting intensity is one. When the confining
potential $V$ is quadratic, e.g., $V(\lambda)=\lambda^2$, this point
process converges in distribution as $N\to\infty$ to a random point
process known as the $\mathrm{sine}_\beta$ process. This was shown
for every $\beta>0$ by Valk\'o and
Vir\'ag~{\cite{valkoViragBrownianCarousel}}.
Let $0=X_0<X_1<\cdots$ denote the nonnegative points of the
unit-intensity $\mathrm{sine}_\beta$ process under its Palm law
(that is, as seen from a typical point){~\cite{VV}}. Define
\[
    G_\beta := X_1-X_0.
\]
We call its law the $\beta$-Gaudin--Mehta distribution.

Under our standing assumptions, if $\beta\geq1$ and $V\in C^4(\mathbb R)$,
then $g_{i_N}^N\xrightarrow{\mathrm d}G_\beta$ whenever $i_N/N$ stays
bounded away from $0$ and $1$. This follows from the Gaussian bulk limit
of Valk\'o and Vir\'ag~{\cite{valkoViragBrownianCarousel}} and the
fixed-index gap universality theorem of Erd\H{o}s and Yau~{\cite{EY}},
in the $C^4$ setting of Bourgade, Erd\H{o}s and
Yau~{\cite[Theorem~2.5(ii)]{BEY}}, using the equivalence of the two
normalisations above. {For $0<\beta<1$, the normalised forward gap from a typical particle of the circular $\beta$-ensemble likewise converges to $G_\beta$ \cite{VV}, and its law is log-concave by log-concavity of the ordered-angle density relative to that particle.}

{For $\beta=2$,} the
$\mathrm{sine}_2$ process is determinantal with correlation kernel
$K(x,y):=\frac{\sin(\pi(y-x))}{\pi(y-x)}$, and accordingly the
distribution function of $G_2$ has an explicit expression in terms of
a Fredholm determinant {\cite{gaudin,tao2}}. The $\mathrm{sine}_2$ process occurs as the marginal
on a single thread of Boutillier's bead model~{\cite{boutillier}}, which
enjoys a tractable integrable structure studied further
in~{\cite{johnston}}.

{Tao~\cite{tao2} proved convergence to $G_2$ for a prescribed bulk GUE gap. The distinction between a prescribed index and a prescribed energy matters here: by Gustavsson's central limit theorem~\cite{gustavsson}, a bulk GUE eigenvalue fluctuates on the scale $\sqrt{\log N}/N$, a factor of $\sqrt{\log N}$ larger than the typical spacing.}

For the stationary unit-intensity $\mathrm{sine}_\beta$ process, the
probability that a fixed interval of length $h$ is empty is
$\mathbb E[(G_\beta-h)_+]$. Valk\'o and
Vir\'ag~{\cite{valkoViragLargeGaps}} obtained sharp asymptotics for this
probability{, whose logarithm is asymptotic to $-\beta\pi^2h^2/16$ as $h\to\infty$.}

{In the finite-$N$ GUE, Ben Arous and Bourgade~\cite{BAB} showed that the largest gap in a fixed bulk interval has order $\sqrt{\log N}/N$, and Feng and Wei~\cite{FW} determined its limiting fluctuations. An upper bound of this order also follows from Theorem~\ref{thm:gaussian} by a union bound over the bulk indices.}

\medskip

Motivated by questions about random hives with GUE boundary data,
Tao~\cite{tao} studied the tails and moments of individual GUE gaps.
At finite $N$, one may ask for an upper bound on the probability that a
prescribed bulk gap is large, uniformly in both $N$ and the index $i$.
Tao~\cite[Theorem~1.1]{tao} proved that the semicircle-normalised gap
$\widetilde g_i^N$ satisfies
\[
    \mathbf{P}\bigl(\widetilde g_i^N\geq h\bigr)
    \leq C_\delta e^{-h/4},
    \qquad 0<h\ll\log\log N,
\]
uniformly for $\delta N\leq i\leq(1-\delta)N$. Tao also obtained bounds
on all positive moments, a lower-tail estimate for a single gap, and
rigidity estimates for sums of consecutive gaps. His proof uses the
determinantal structure of the GUE eigenvalues. Immediately after
Theorem~1.1, Tao {predicted that the upper tail would exhibit}
Gaussian decay in $h$. We prove that this is indeed the case in Theorem~\ref{thm:GUE}.

\subsection{Log-concavity and the goal of the present article}

The purpose of the present paper is to bring to bear the simple but
powerful tool of \emph{log-concavity} to obtain concentration
bounds for the $\lambda_i^N$ and their gaps. We observe that if $V$ is
convex, then the probability density function in \eqref{eq:betaens} is
log-concave, in the sense that it can be written as $e^{-W}$ where
$W:\mathbb{R}^N\to\mathbb{R}\cup\{+\infty\}$ is a convex function.
Since log-concavity is preserved under linear maps, each
eigenvalue and each gap also has a log-concave law. Using log-concavity
alone, we are able to give simple proofs of several absolute
bounds for functionals of the $\lambda_i^N$, without
recourse to heavier-duty techniques in integrable probability such as orthogonal polynomials or determinantal processes. For the
Gaussian upper-tail bound we also exploit the logarithmic interaction
between the particles. 

We emphasise however, that the idea of using log-concavity to study the {eigenvalues of random matrices} is not new. 
Log-concavity has also been used to study random hives,
representation-theoretic objects encoding the eigenvalues of sums of
Hermitian matrices. Narayanan, Sheffield and Tao~\cite{NST} used the
log-concavity of the joint law of two independent GUE minor processes
in their proof of concentration for random hives with GUE boundary
conditions. Concentration bounds obtained by Tao~\cite{tao} have
subsequently been used by Narayanan~\cite{narayanan} to
prove convergence of these random hives to a continuum hive. 
The log-concavity of the uniform measure on a
Gelfand--Tsetlin polytope plays a central role in recent work of the
author with Prochno~\cite{JP}, where a large deviation principle {for random} Gelfand--Tsetlin patterns is established, with
applications to problems in free probability.
See
also~\cite{NS} for related developments.

\medskip
The theory of log-concave measures forms part of a broader body of work
surrounding the Brunn--Minkowski and Pr\'ekopa--Leindler inequalities
and the geometry of convex bodies, and has been a central theme in
asymptotic geometric analysis over the past several decades. Background
and further references may be found in the monographs
\cite{brazitikos, AGA}.
\subsection{Overview} 

The remainder of the article consists of two further sections: In Section \ref{sec:results} we state our main results, and in Section \ref{sec:proofs} we provide their proofs. 

\section{Main results} \label{sec:results}

The main idea of the present article is to bring the basic tools of log-concavity to study the points and gaps between points of $\beta$-ensembles. 

In this direction, recall that a probability measure $\mu$ on $\mathbb{R}^N$ is said to be log-concave if for every $0<\lambda<1$ and all Borel subsets $A$ and $B$ of $\mathbb{R}^N$, we have
\begin{align} \label{eq:logconcave}
\mu( \lambda A + (1-\lambda) B) \geq \mu(A)^\lambda \mu(B)^{1-\lambda},
\end{align}
where $\lambda A + (1-\lambda)B$ is the set of vectors in $\mathbb{R}^N$ of the form $\lambda a + (1-\lambda)b$, where $a \in A$ and $b \in B$. If $\mu$ has a probability density function $f:\mathbb{R}^N \to [0,\infty)$ with respect to Lebesgue measure, then $\mu$ is log-concave if and only if its probability density function may be written $f = e^{ - W}$ where $W:\mathbb{R}^N \to \mathbb{R}\cup\{+\infty\}$ is convex. 

We begin with a simple observation:
\begin{obs} \label{obs1}
If $p_N$ is a probability density function of the form in \eqref{eq:betaens} and $V$ is convex, then $p_N$ is a log-concave density function. 
\end{obs}

Indeed, the function 
\[
-\log p_N=-\log C_N-\beta\sum_{i<j}\log(\lambda_j-\lambda_i)
                  +\tfrac{N\beta}{2}\sum_i V(\lambda_i)
\]
is convex on $\mathbb{R}^N$, taking the convention that $-\log p_N = +\infty$ on the complement $\mathbb{R}^N - \mathbb{R}^{N, \uparrow}$ of the set $\mathbb{R}^{N, \uparrow}$ of vectors with nondecreasing coordinates.

Random vectors with log-concave laws are guaranteed to satisfy a range of excellent stability and concentration properties \cite{brazitikos}. 
Moreover, the attribute of log-concavity is preserved under a large class of transformations. Indeed, we have the following well-known observation, which is an immediate consequence of the fact that the inequality \eqref{eq:logconcave} is preserved under taking linear maps:

\begin{proposition} \label{prop:proj}  \cite{LV, prekopa, leindler}
Let $X$ be a random vector in $\mathbb{R}^N$ with a log-concave law and let $T:\mathbb{R}^N \to \mathbb{R}^n$ be a linear map. Then the image $T(X)$ has a log-concave law on $\mathbb{R}^n$. 
\end{proposition}

We will primarily be interested in applying the case $n=1$ of Proposition \ref{prop:proj} to Observation \ref{obs1}. Indeed, we can conclude immediately by letting $T$ be the $i^{\text{th}}$ projection map $T(\lambda_1,\ldots,\lambda_N) = \lambda_i$, that the marginal law of $\lambda_i^N$ under $p_N$ in \eqref{eq:betaens} is a log-concave probability distribution on $\mathbb{R}$. Of course, so is the map $T:\mathbb{R}^N \to \mathbb{R}$ sending $(\lambda_1,\ldots,\lambda_N)$ to a scalar multiple of $\lambda_{i+1}-\lambda_i$, and accordingly we have the following corollary:

\begin{obs} \label{obs2}
Let $V$ be convex, and let $(\lambda_1^N,\ldots,\lambda_N^N)$ be distributed according to the probability density function in \eqref{eq:betaens}. Then for any $1 \leq i \leq N-1$, the expectation-normalised gap
\begin{align*}
g^N_i := \frac{ \lambda_{i+1}^N - \lambda_i^N }{ \mathbb{E}[ \lambda_{i+1}^N - \lambda_i^N ] }
\end{align*}
is log-concave, has unit mean, and has law supported on $(0,\infty)$. 

\end{obs}

It transpires that log-concave random variables {with unit mean and} laws supported on $(0,\infty)$ are extremely well-behaved:

\begin{proposition} [Variant of results in \cite{MM, LV}] \label{prop:main2}
If $X$ is a unit-mean random variable with a log-concave density supported on $(0,\infty)$, then $X$ satisfies the following properties:
\begin{enumerate}[label=(\roman*)]
\item (Concentration inequality) For all $h > 0$ we have $\mathbf{P}( X \geq h ) \leq e^{1-h}$. 
\item (Moment bound) For all $p \geq 1$ we have $\mathbf{E}[ X^p] \leq p!$. 
\item (Lower tail bound) For all $h > 0$ we have $\mathbf{P}( X \leq h ) \leq 2h$. 
\item (Gr\"unbaum's inequality) We have $\mathbf{P}( X \geq 1) \geq 1/e$. 
\end{enumerate}
\end{proposition}

Parts (ii) and (iv) of Proposition \ref{prop:main2} can both be found in Section 2 of \cite{brazitikos} (see also \cite[Lemma 5.4]{LV}). In Section \ref{sec:proofs} we show how parts (i) and (iii) of Proposition \ref{prop:main2} follow from results on convex orderings between log-concave random variables \cite{MM}.

Combining Observation \ref{obs2} and Proposition \ref{prop:main2} we have the following:

\begin{cor} \label{cor2}
For all $1 \leq i \leq N-1$, the expectation-normalised gaps $g_i^N$ between points of the $\beta$-ensemble in \eqref{eq:betaens} satisfy the inequalities (i)-(iv) of Proposition \ref{prop:main2}. 
\end{cor}

We discuss in Section \ref{sec:taosharpen} how Corollary \ref{cor2} sharpens and {further specifies} parts of Theorem 1.1 of \cite{tao}.

\medskip
Recall from the introduction the notion of a $\beta$-Gaudin-Mehta distributed random variable $G_\beta$, which measures the typical gap between consecutive points in a $\mathrm{sine}_\beta$ point process. Equivalently, under certain conditions on the background potential $V$, $G_\beta$ is the limiting law of the normalised gap $g_{i_N}^N$ for a sequence $i_N$ lying in the bulk. 

We will note in the sequel that if $(X_N)_{N\geq 1}$ {is a sequence of nonnegative} log-concave random variables, {each with unit mean}, and with the property that they converge in distribution as $N \to \infty$, then the limiting distribution must itself be log-concave. Combining this fact with {the bulk gap convergence discussed in the introduction} we obtain the following corollary:

\begin{cor}
{For $\beta > 0 $,} the $\beta$-Gaudin-Mehta distribution is log-concave.
\end{cor}

The anonymous preprint \cite{anon} (which itself cites an earlier version of the present article) obtains the stronger conclusion that, writing $p^\beta$ for the {density of $G_\beta$}, the function $s\mapsto p^\beta(s)/s^\beta$ is log-concave on $(0,\infty)$.

\medskip
We now prepare the statement of our final result. For this result we will require a slightly stronger-than-usual notion of bulk index, that solicits control of the expected span of the particles on either side of the gap.

A Gaussian bound on mesoscopic scales requires control of the expected span of the particles on either side of the gap. 

\begin{df}
Write $d_i^N=\mathbb E[\lambda_{i+1}^N-\lambda_i^N]$. We say that {$L_\delta\geq1$ is an expected gap constant} for the $\delta$-bulk if for all $N \geq 1$ we have 
\begin{align*}
{\frac{1}{L_\delta N}} \leq d_i^N \leq {\frac{L_\delta}{N}} \qquad \text{for all $\delta N \leq i \leq (1-\delta)N$}.
\end{align*}
\end{df}

{We prove the following proposition in Appendix~\ref{app:expected-gaps} using the bulk gap universality theorem of Bourgade, Erd\H{o}s and Yau~\cite[Theorem~2.5(ii)]{BEY}.}
\begin{proposition}\label{prop:expected-gaps}
If $V$ is $\mathcal{C}^4$ and $\beta \geq 1$, expected gap constants exist for all $\delta \in (0,1/2)$.  
\end{proposition}

\begin{thm}\label{thm:gaussian}
{Suppose there exist expected gap constants} for all $\delta > 0$. 
 Then there are constants $C_\delta, c_\delta$ (depending on $\beta$ and the background potential $V$ but not on $N$) such that for all $N \geq 1$ and $\delta N \leq i \leq (1-\delta)N$ we have 
\begin{align*}
\mathbb{P}( g_i^N \geq g ) \leq C_\delta e^{ - c_\delta \min \{ g^2, gN \}} \qquad \text{for all $g \geq 0$}.
\end{align*}
\end{thm}

{Let us remark that when considering events of the form $g_i^N \geq aN$ for fixed $a>0$ and $i$ in the bulk, the unnormalised gap $\lambda_{i+1}^N-\lambda_i^N$ is bounded below by a positive constant, and we are considering an event where the macroscopic shape of the measure is changed. This is the setting of the large-deviation principle for empirical measures of beta ensembles \cite[Theorem~2.6.1]{AGZ}; see also Guionnet and Zeitouni \cite{GZ} for related concentration bounds.}

We prove Theorem {\ref{thm:gaussian}} using a clever trick (as mentioned in the opening, that was suggested to the author by AI) that manipulates the well-known Pr\'ekopa-Leindler inequality for log-concave functions.

We will show later that by using the fact that {$2N\mathrm{tr}(X^2)$ has the $\chi^2_{N^2}$ distribution for a GUE random matrix $X$}, it is straightforward to replace the $\min\{g^2, Ng \}$ term with $g^2$ in the GUE case:
\begin{thm}\label{thm:GUE}
The bulk eigengaps in GUE random matrices have Gaussian tails. More specifically, let $\beta = 2$ and $V(\lambda) = \lambda^2$. Then there are constants $C_\delta, c_\delta$ such that for all $N \geq 1$ and $\delta N \leq i \leq (1-\delta)N$ we have 
\begin{align*}
\mathbb{P}( g_i^N \geq g ) \leq C_\delta e^{ - c_\delta g^2 } \qquad \text{for all $g \geq 0$}.
\end{align*}
\end{thm}

That completes the statements of our main results. 

\section{Proofs of our main results} \label{sec:proofs}
\subsection{Proof of Proposition \ref{prop:main2}}

In this section we outline how Proposition \ref{prop:main2} follows from various properties in the literature.

As noted following its statement, {Parts (ii) and (iv)} of Proposition \ref{prop:main2} are from \cite{brazitikos, LV}. We now explain how to obtain {parts (i) and (iii) of Proposition} \ref{prop:main2} using results of Marsiglietti and Melbourne \cite{MM} on convex orderings between log-concave random variables. 
Given two probability measures $\mu$ and $\nu$ on the real line, we say that $\mu$ is dominated by $\nu$ in the convex order, and write $\mu \preceq_{\mathrm{cx}} \nu$, if the inequality
\begin{align*}
\int_{-\infty}^\infty \varphi(x)\,\mu(\mathrm{d}x) \;\leq\; \int_{-\infty}^\infty \varphi(x)\,\nu(\mathrm{d}x)
\end{align*}
holds for every convex function $\varphi:\mathbb{R}\to\mathbb{R}$ for which the relevant integrals exist. By considering the functions $\varphi(x)=x$ and $\varphi(x)=-x$, both of which are convex, we note that $\mu$ and $\nu$ must necessarily have the same expectation.

Suppose that $\mu$ is absolutely continuous with respect to $\nu$. We say that $\mu$ is log-concave relative to $\nu$ if the Radon--Nikodym derivative $\mathrm{d}\mu/\mathrm{d}\nu$ is log-concave. {If $\mu$ and $\nu$ have the same finite mean and $\mathrm{d}\mu/\mathrm{d}\nu$ is log-concave, then} it is necessarily the case that $\mu$ is dominated by $\nu$ in the convex ordering \cite{MM}.
 In particular, one may verify that if $\mu$ is any unit-mean log-concave measure supported on $[0,\infty)$, then the Radon-Nikodym derivative with respect to the standard exponential measure $\nu_{\mathrm{exp}}(\mathrm{d}x) = e^{-x}\mathrm{1}_{\{ x> 0 \}}\mathrm{d}x$, should it exist, is log-concave. 
 Combining these observations, we have:

 \begin{proposition}
 If $\mu$ is a log-concave measure on $[0,\infty)$ with unit mean, then $\mu$ is dominated in the convex order by the standard exponential measure $\nu_{\mathrm{exp}}$.
 \end{proposition}

We note that we may immediately set $\varphi(x) = x^p$ and use the previous result to obtain Proposition \ref{prop:main2}{(ii)}.

\begin{proof}[Proof of Proposition \ref{prop:main2} (i) and (iii)]

First we prove Proposition \ref{prop:main2} (i). For $h > 0$, consider the ``call option'' type function given by $\varphi(x) = (x - (h-1))_+$. This function is convex and dominates the indicator function $\mathrm{1} \{ x \geq h \}$. Using this fact to obtain the first inequality below, $\mu \prec_{\mathrm{cx}} \nu_{\mathrm{exp}}$ to obtain the second, and then integrating out to obtain the third, we have
\begin{align*}
\mathbf{P}_\mu( X \geq h ) \leq \mathbf{E}_\mu[ \varphi(X) ] \leq \mathbf{E}_{\nu_{\mathrm{exp}}}[ \varphi(X) ] \leq  \int_{h-1}^\infty (x-(h-1)) e^{- x} \mathrm{d}x = e^{1-h},
\end{align*}
completing the proof of Proposition \ref{prop:main2} (i).

To prove Proposition \ref{prop:main2} (iii), instead consider the ``put option'' function $\varphi:\mathbb{R} \to \mathbb{R}$ given by $\frac{1}{h}(2h-x)_+$, which is convex and dominates $\mathrm{1}\{ x \leq h \}$. Arguing similarly we have 
\begin{align*}
\mathbf{P}_\mu( X \leq h ) \leq \mathbf{E}_\mu[ \varphi(X) ] \leq \mathbf{E}_{\nu_{\mathrm{exp}}}[ \varphi(X) ] \leq \frac{1}{h} \int_0^{2h}(2h-x) e^{-x} \mathrm{d}x \leq 2h,
\end{align*}
completing the proof of Proposition \ref{prop:main2} (iii).
\end{proof}

\subsection{Semicircle normalised GUE gaps} \label{sec:taosharpen}

In Corollary \ref{cor2} we saw that the expectation-normalised gaps $g_i^N$ of a $\beta$-ensemble with a convex background potential automatically satisfy the bounds in Proposition \ref{prop:main2}. Tao \cite{tao} considers the semicircle-renormalised eigengaps $\tilde{g}_i^N := N \rho(Q(i/N))(\lambda_{i+1}^N - \lambda_i^N)$, proving the following:

\begin{thm}[Theorem 1.1 of \cite{tao}] \label{thm:tao}
Whenever $\delta \leq i/N \leq 1-\delta$ we have
\begin{enumerate}[label=(\roman*)]
\item (Concentration inequality) For all $0 <g \leq C \log \log N$ we have $\mathbf{P}( \tilde{g}_i^N \geq g ) \leq C_\delta e^{ - g/4}$. 
\item (Moment bound) For all $p \geq 1$ we have $\mathbf{E}[ (\tilde{g}_i^N)^p ] \leq C_{\delta,p}$. 
\item (Lower tail bound) For all $g\in (0,1/2) $ we have $\mathbf{P}( \tilde{g}_i^N \leq g ) \leq C_\delta g^{2/3} \log (1/g)$.
\item (Local eigenvalue rigidity) For any natural number $0 < m \leq \log^{O(1)}N$ and $\alpha > 0$ we have 
\begin{align*}
\mathbf{P}( |\tilde{g}^N_i + \cdots + \tilde{g}^N_{i+m-1} - m | > \alpha ) \leq C_\delta \frac{\log^{4/3}(2+m)}{\alpha^2}
\end{align*}
and
\begin{align*}
 \mathbf{E}[ |\tilde{g}^N_i + \cdots + \tilde{g}^N_{i+m-1} - m |^2]  \leq C_\delta \log^{7/3}(2+m).
\end{align*}
\end{enumerate}
\end{thm}

As mentioned above, up to the issue of the different normalisations between $g_i^N$ and $\tilde{g}_i^N$, Corollary \ref{cor2} sharpens Theorem \ref{thm:tao} (i)-(iii). Using bounds from a different article \cite{tao2} of Tao, we can spell out here exactly how the normalisations coincide:

\begin{proposition}[\cite{tao2}] \label{prop:expas}
With $o(1)$ terms that go to zero uniformly as $i,N \to \infty$ whenever $i/N$ is constrained to $[\delta,1-\delta]$, we have 
\begin{align*}
\mathbf{E}[\tilde{g}_i^{(N)}] = (1+o(1))
\end{align*}
\end{proposition}
\begin{proof}
The quantity {$\tilde g_i^{(N)}$} is precisely the random variable $X$ defined at the beginning of Section 4 of \cite{tao2}. By setting $s = 0$ in equation (23) of that article, we see in particular that
\begin{align*}
\mathbf{E}[ {\tilde g_i^{(N)}} ] = \int_0^\infty y p(y)\mathrm{d}y + o(1),
\end{align*}
where $p(y)$ is the pdf of the Gaudin-Mehta law, and where the $o(1)$ converge to zero as $N \to \infty$, uniformly for all $i$ with $\delta \leq i/N \leq 1-\delta$. It can be seen, e.g.\ by setting $y= 0$ in equation (4) of \cite{tao2}, that $\int_0^\infty y p(y) \mathrm{d}y = 1$. That completes the proof.
\end{proof}

We close this section by mentioning that log-concavity tools can be used in conjunction with Theorem \ref{thm:tao} (iv) to prove the following: 

\begin{proposition} \label{prop:multigap} Let $0 < m \leq \log^{O(1)}N$ and $\delta \leq i/N \leq 1-\delta$. Then
\begin{align*}
\mathbf{P} ( |\tilde{g}^N_i + \cdots + \tilde{g}^N_{i+m-1} - m| > \alpha ) \leq \exp \left\{ 1 - c_\delta \frac{\alpha}{\log^{7/6}(2+m)}  \right\}. 
\end{align*}
\end{proposition}

To prove Proposition \ref{prop:multigap}, we appeal to another well-known bound, which states that if $X$ is a real-valued random variable with a log-concave distribution, then
\begin{align} \label{eq:tail}
\mathbf{E}[|X|^2] \leq 1 \implies \mathbf{P}( |X| > t ) \leq e^{1 - t } \qquad \text{for all $t > 0$};
\end{align}
see Lemma 5.7 of Lov\'asz and Vempala \cite{LV}.
We are now able to prove Proposition \ref{prop:multigap} by combining the variance bound in Theorem \ref{thm:tao} (iv) with this bound:

\begin{proof}[Proof of Proposition \ref{prop:multigap}]
{For brevity, write $g_j=\tilde g_j^N$.} Set
\begin{align*}
\sigma_{i,m}^2  =  \mathbf{E}[ |g_i + \cdots + g_{i+m-1} - m |^2].
\end{align*}
Theorem \ref{thm:tao} (iv) states that provided $0 < m \leq \log^{O(1)}N$ we have 
\begin{align} \label{eq:cris}
\sigma_{i,m} \leq C_\delta \log^{7/6}(2+m).
\end{align}
Now the random variable
\begin{align*}
X_{i,m} := \frac{ g_i + \cdots + g_{i+m-1} - m  }{\sigma_{i,m} }
\end{align*}
plainly satisfies $\mathbf{E}[|X_{i,m}|^2] = 1$. Moreover, by Proposition \ref{prop:proj}, $X_{i,m}$ has a log-concave distribution. It follows from \eqref{eq:tail} that
\begin{align*}
\mathbf{P}\left( |g_i + \cdots + g_{i+m-1} - m | > \alpha \right) = \mathbf{P}( |X_{i,m}| > \alpha/\sigma_{i,m} ) \leq e^{1 - \alpha/\sigma_{i,m}}.
\end{align*}
The result now follows from using the upper bound in \eqref{eq:cris}.
\end{proof}

\subsection{Proof of Theorem \ref{thm:gaussian}}

Our proof of Theorem \ref{thm:gaussian} uses the Pr\'ekopa-Leindler inequality, which says the following. Let $f,g,h:\mathbb{R}^N \to [0,\infty)$ be integrable log-concave functions satisfying, for some $\alpha \in (0,1)$, the bound 
\begin{align*}
{h( \alpha \lambda + (1-\alpha) \mu) \geq f(\lambda)^\alpha g(\mu)^{1-\alpha} \qquad \lambda,\mu \in \mathbb{R}^N.}
\end{align*}
Then we have 
\begin{align*}
\int_{\mathbb{R}^N} h(\lambda) \mathrm{d}\lambda \geq \left( \int_{\mathbb{R}^N} f(\lambda) \mathrm{d}\lambda \right)^\alpha \left( \int_{\mathbb{R}^N} g(\lambda) \mathrm{d}\lambda \right)^{1-\alpha};
\end{align*}
See \cite{AGA} for further information.

We will appeal to the following fact:
\begin{lemma} \label{lem:survival}
Let $f$ be a log-concave probability density function on $(0,\infty)$. Then $S(x) := \int_x^\infty f(u)\mathrm{d}u$ is log-concave.
\end{lemma}

We are now equipped to prove Theorem \ref{thm:gaussian}. In fact, {here we provide} a slightly more specific statement that will describe how the constants depend on $\delta, L_\delta$ and $N$.
\begin{thm}\label{thm:gaussian2}
Let $\delta \in (0,1/4)$.
{Let $L=L_\delta\geq\beta$ be an expected gap constant for the $\delta$-bulk.} Then for all integers $N \geq 1$ and all indices $i$ with $2 \delta N \leq i \leq (1-2\delta)N$ we have
\begin{align} \label{eq:brucker0}
\mathbb{P}( g_i^N \geq g ) \leq e^{c \beta/4 }
\exp\left\{-c' \beta\min\left\{\frac{g^2}{L^4},
\frac{\delta N g}{L^2}\right\}\right\},\qquad g\geq 0,
\end{align}
{where {$c,c'>0$ are universal}.}
\end{thm}

\begin{proof}
Recall that we are assuming that $L_\delta, \delta > 0$ is chosen so that
$\frac{1}{L_\delta}\frac{1}{N} \leq d_j^N :=
\mathbb{E}[\lambda_{j+1}^N - \lambda_j^N] \leq L_\delta \frac{1}{N}$
for all $\delta N \leq j \leq (1-\delta)N$.
The statement of Theorem \ref{thm:gaussian} solicits an upper bound on
$\mathbb{P}( \lambda_{i+1}^N - \lambda_i^N \geq d_i^N g )$.

Recall that the probability density function
$p_N:\mathbb{R}^N \to [0,\infty)$ in \eqref{eq:betaens} is
log-concave, and accordingly, it is the case that for all $\lambda,\mu$ in $\mathbb{R}^{N,\uparrow}$ (for which
$p_N(\lambda),p_N(\mu)>0$) we will have
\begin{align} \label{eq:brahms}
\log p_N( (\lambda+ \mu)/2) - \frac{1}{2}\log p_N(\lambda)
- \frac{1}{2} \log p_N(\mu) \geq 0.
\end{align}
The idea of the present proof is to show that if $\lambda$ is a
`fairly typical' point in $\mathbb{R}^{N,\uparrow}$ (as far as $p_N$
is concerned), whereas $\mu$ is a point in $\mathbb{R}^{N,\uparrow}$
for which {the} $i^{\text{th}}$ gap is unusually large, then we can bound
the quantity in \eqref{eq:brahms} away from zero in a way that
depends on $g$, and thereafter apply the case $\alpha = 1/2$ of
the Pr\'ekopa-Leindler inequality.

For the remainder of the proof, fix $i,N$ and abbreviate
$d=d_i^N$ and $L=L_\delta$. For real $g > 0$
and integer $k \geq 0$ {with $k\leq\delta N$}, consider the two subsets of
$\mathbb{R}^{N,\uparrow}$ given by
\begin{align*}
A_g &:= \{ \mu \in \mathbb{R}^{N,\uparrow} :
\mu_{i+1} - \mu_i \geq g d\}, \\
B_k &:= \{ \lambda \in \mathbb{R}^{N,\uparrow} :
\lambda_{i+k} - \lambda_{i-k+1}
\leq 4L k / N \},
\end{align*}
where we take the convention $B_0 = \mathbb{R}^{N,\uparrow}$ if $k = 0$.
If $2\delta \leq i/N \leq 1-2\delta$ and
${1\leq{}}k \leq \delta N$, then by Markov's inequality and the fact that
$
\mathbb{E}[\lambda_{i+k}-\lambda_{i-k+1}]
=\sum_{j=i-k+1}^{i+k-1}d_j^N
\leq 2Lk/N$, we have $p_N(B_k) \geq 1/2$.
Of course, $p_N(A_g)$ is the object we are trying to bound.

By the convexity of $V$, the contribution from the background
potential to the quantity in \eqref{eq:brahms} is nonnegative,
and considering only the contribution from the Vandermonde
component we have
\begin{align} \label{eq:brahms2}
\log p_N( (\lambda+ \mu)/2) - \frac{1}{2}\log p_N(\lambda)
- \frac{1}{2} \log p_N(\mu)
\geq \beta \sum_{1 \leq p < q \leq N}
\log\frac{a_{pq}+b_{pq}}{2\sqrt{a_{pq}b_{pq}}},
\end{align}
where $a_{pq}=\mu_q - \mu_p$ and $b_{pq}=\lambda_q-\lambda_p$.

Note that for $a,b > 0$, the quotient
$\frac{a + b}{2 \sqrt{ab}} \geq 1$, with equality if and only
if $a = b$. Moreover, whenever $a/b > 1$, this quotient is
monotone increasing in $a/b$.

Consider that for $p \in \{i-k+1,\ldots,i\}$ and
$q \in \{i+1,\ldots,i+k\}$, for $\lambda \in B_k$ we have
the upper bound
\begin{align} \label{eq:orange1}
b_{p,q} := \lambda_q - \lambda_p
\leq \lambda_{i+k} - \lambda_{i-k+1}
\leq 4 L k/N.
\end{align}
Conversely, for $p \in \{i-k+1,\ldots,i\}$ and
$q \in \{i+1,\ldots,i+k\}$, if $\mu \in A_g$ we have
the lower bound
\begin{align} \label{eq:orange2}
a_{p,q} := \mu_q - \mu_p
\geq \mu_{i+1} - \mu_i
\geq g \frac{1}{L} \frac{1}{N}.
\end{align}

Combining \eqref{eq:orange1} and \eqref{eq:orange2}, we see
that if
$k \leq \frac{1}{8}\frac{1}{L^2}g \wedge \delta N$,
then for all $p \in \{i-k+1,\ldots,i\}$ and
$q \in \{i+1,\ldots,i+k\}$ we have $a_{p,q}/b_{p,q} \geq 2$,
and thus for all $p,q$ in this range we have
\begin{align} \label{eq:orange3}
\log\frac{a_{pq}+b_{pq}}{2\sqrt{a_{pq}b_{pq}}}
\geq \log \frac{2+1}{2 \sqrt{2}} = c_0 > 0.
\end{align}

Define
\begin{align*}
k_{g,\delta} := \left\lfloor \min\left\{ \frac{1}{8 L^2} g, \delta N\right\} \right\rfloor,
\end{align*}
where for real $x$, $\lfloor x \rfloor$ denotes the largest integer not exceeding $x$.
By considering the contribution to \eqref{eq:brahms2} coming from
the $k_{g,\delta}^2$ pairs of the form
$p \in \{i-k_{g,\delta}+1,\ldots,i\}$ and
$q \in \{i+1,\ldots,i+k_{g,\delta}\}$, using
\eqref{eq:orange3} we see that whenever $\mu \in A_g$ and
$\lambda \in B_{k_{g,\delta}}$
with $p_N(\mu),p_N(\lambda)>0$ we have
\begin{align} \label{eq:brahms3}
\log p_N( (\lambda+ \mu)/2) - \frac{1}{2}\log p_N(\lambda)
- \frac{1}{2} \log p_N(\mu)
\geq c_0 \beta k_{g,\delta}^2.
\end{align}
We are now ready to use the Pr\'ekopa-Leindler inequality. Define three integrable functions $a,b,c:\mathbb{R}^N \to [0,\infty)$ by
\begin{align*}
a(\lambda) &= p_N(\lambda)\mathrm{1}\{ \lambda \in A_g\}, \\
b(\lambda)
&:= p_N(\lambda)\mathrm{1}\{ \lambda \in B_{k_{g,\delta}}\}, \\
c(\lambda) &:= e^{-c_0\beta k_{g,\delta}^2} p_N(\lambda).
\end{align*}
By \eqref{eq:brahms3} we have
\begin{align*}
c\left( \frac{\mu + \lambda}{2} \right)
\geq a(\mu)^{1/2}
b(\lambda)^{1/2}.
\end{align*}
It follows from the Pr\'ekopa-Leindler inequality that
\begin{align*}
e^{-c_0\beta k_{g,\delta}^2}
&= \left( \int_{\mathbb{R}^N} c(\lambda)\mathrm{d}\lambda \right) \\
&\geq \left( \int_{\mathbb{R}^N} a(\lambda)
\mathrm{d}\lambda\right)^{1/2}
\left( \int_{\mathbb{R}^N}
b(\lambda)
\mathrm{d}\lambda\right)^{1/2} =\sqrt{p_N(A_g)p_N(B_{k_{g,\delta}})}.
\end{align*}

Using the fact that $p_N(B_{k_{g,\delta}}) \geq 1/2$, rearranging we obtain
\begin{equation*}
S(g) := p_N(A_g) \leq 2 e^{-2c_0 \beta k_{g,\delta}^2} \qquad g > 0.
\end{equation*}
For $g\geq 2$, Markov's inequality also gives
$S(g)\leq 1/g\leq 1/2$. Taking the geometric mean of these two bounds gives
$S(g)\leq e^{-c_0\beta k_{g,\delta}^2}$ for ${g}\geq 2$.

{Since $L\geq 1$, we have $8L^2\geq 2$. Suppose first that $\delta N\geq 1$.}
Let $g_0 = 8 L^2 \delta N$. Then
using $\lfloor x\rfloor\geq x/2$ for $x\geq 1$, we have
\begin{align} \label{eq:brucker}
S(g) \leq \exp\left\{-  \frac{ c_1 \beta }{L^4} g^2
\right\},
\qquad 8L^2\leq g\leq g_0.
\end{align}
{Here $c_1=c_0/256$.}

As for the case $g \geq g_0$, we note from Lemma \ref{lem:survival} that $S(g)$ is log-concave and satisfies $S(0) = 1$, from which it follows that $S(g) \leq S(g_0)^{g/g_0}$ when $g \geq g_0$. Thus we have
\begin{align} \label{eq:brucker2}
S(g) \leq S(g_0)^{g/g_0} \leq \exp\left\{-  \frac{ c_1 \beta }{L^4} g_0 g \right\}\qquad g \geq g_0.
\end{align}
Finally, note that $c_1 \beta (g/L^2)^2 \leq 64 c_1 \beta$ for $g \in (0,8L^2)$, so that combining \eqref{eq:brucker} and \eqref{eq:brucker2} we have the following bound valid for all $g > 0$:
\begin{align} \label{eq:brucker3}
S(g) \leq \exp\left\{ c_2 \beta -\frac{ c_1 \beta }{L^4} (g^2 \wedge g g_0 )\right\}\qquad g \geq 0.
\end{align}
{Here $c_2=64c_1=c_0/4$.}
Being generous with the constants we have
\begin{align} \label{eq:brucker4}
S(g)\leq e^{c \beta/4 }
\exp\left\{-c' \beta\min\left\{\frac{g^2}{L^4},
\frac{\delta N g}{L^2}\right\}\right\},\qquad g\geq 0,
\end{align}
{with $c=c_0$ and $c'=c_1$. Finally, if $\delta N<1$, the
same small-$g$ estimate applies for $g<8L^2$, while for $g\geq 8L^2$
the basic bound $S(g)\leq e^{1-g}\leq e^{-g/2}$ gives
\eqref{eq:brucker4}, {since $L\geq\max\{1,\beta\}$ implies}
{$c'\beta\delta N/L^2\leq c'<1/2$}.}
\end{proof}

Finally, we prove Theorem \ref{thm:GUE}.

\begin{proof}[Proof of Theorem \ref{thm:GUE}]
In the $\delta$-bulk, $d_i^N\geq 1/(L_\delta N)$. It follows that if
$g_i^N\geq g$ then $\lambda_{i+1}^N-\lambda_i^N\geq g/(L_\delta N)$.
Thus either $\lambda_i^N\leq -g/(2L_\delta N)$ or
$\lambda_{i+1}^N\geq g/(2L_\delta N)$. In either case at least
$\delta N$ eigenvalues have absolute value at least $g/(2L_\delta N)$,
so $2N\mathrm{tr}(X^2)\geq c_\delta g^2$.
Since $2N\mathrm{tr}(X^2)$ has the $\chi^2_{N^2}$ distribution,
its standard tail bound gives $\mathbb P(g_i^N\geq g)\leq e^{-c'_\delta g^2}$
for $g\geq K_\delta N$ with $K_\delta$ sufficiently large.
This bound can be patched with the case $g \leq K_\delta N$ using Theorem \ref{thm:gaussian}.
\end{proof}

\appendix
\begingroup
\section{Expected gaps in the bulk}\label{app:expected-gaps}

In our proof of Proposition \ref{prop:expected-gaps}, we will note the following result:

\begin{lemma} \label{lem:expectation}
Let $(X_N)_{N \geq 1}$ be a sequence of log-concave random {variables supported on $(0,\infty)$} that converge in distribution to a random variable $X$. Then $\mathbb{E}[X_N] \to \mathbb{E}[X]$. 
\end{lemma}
\begin{proof}
{Write $S_N(t)=\mathbb P(X_N\geq t)$. We may choose  $R>0$ such that $S_N(R)\leq 1/2$ for every $N \geq 1$. Since $S_N(0)=1$, log-concavity of $S_N$ immediately tells us that $S_N(t)\leq 2^{-t/R}$ for $t\geq R$. Now $S_N(t)\to\mathbb P(X\geq t)$ for almost every $t$, so dominated convergence in $\mathbb E[X_N]=\int_0^\infty S_N(t)\,\mathrm dt$ gives the result.}
\end{proof}

\begin{proof}[Proof outline for Proposition~\ref{prop:expected-gaps}]
{By ~\cite[Theorem~2.3 and Remark~5.1]{EY}
(see also~\cite[Theorem~2.5(ii)]{BEY}), under the assumptions of Proposition \ref{prop:expected-gaps}, the density normalised-gaps converge in distribution {to $G_\beta$ along every sequence of bulk indices, using the Gaussian bulk limit \cite{valkoViragBrownianCarousel} as above}. By {\cite[Theorem~4.3 and Corollary~5.1]{GP}}, the equilibrium density $\rho$ is {continuous and} strictly positive on the interval $[Q(\delta),Q(1-\delta){]}$. These two results, together with Lemma \ref{lem:expectation} can be used to prove Proposition \ref{prop:expected-gaps}.}

\end{proof}
\endgroup

\end{document}